\documentclass[12pt]{article}
\usepackage{amsmath,amssymb,amsfonts,amsthm,graphicx,subcaption,xcolor,verbatim}
\usepackage[margin=1in]{geometry}

\theoremstyle{plain}

\newtheorem{theorem}{Theorem}[section]
\newtheorem{proposition}[theorem]{Proposition}
\newtheorem{construction}[theorem]{Construction}
\newtheorem{lemma}[theorem]{Lemma}

\newtheorem{corollary}[theorem]{Corollary}
\newtheorem{question}[theorem]{Question}

\theoremstyle{remark}

\newtheorem*{remark}{Remark}

\title{Embeddings of critical graphs near the Heawood bound}
\author{Timothy Sun\\Department of Computer Science\\San Francisco State University}
\date{}

\newcommand{\Z}{\mathbb{Z}}
\newcommand{\TT}[1]{\underline{\color{red} #1}}

\begin{document}

\maketitle

\begin{abstract}
Complementing a theorem of \v{S}krekovski, we characterize the $(h-1)$-critical graphs embeddable in surfaces of Euler genus at least $5$, where $h$ denotes the Heawood number of the surface. Outside of a few small cases, the bulk of our proof is determining the genus of the join of a complete graph and the 5-cycle. As a byproduct of our proof, we also provide a simpler solution to the minimum triangulations problem for nonorientable surfaces using the theory of current graphs.
\end{abstract}

\section{Introduction}

Given a closed surface of Euler characteristic $\chi$, the \emph{Euler genus} of the surface is defined to be $2-\chi$, i.e., the usual crosscap number for nonorientable surfaces and twice the number of handles for an orientable surface. For short, we say that a graph is \emph{$g$-embeddable} if it is embeddable in the surface(s) of Euler genus $g$. The \emph{Heawood number} $H(g)$ of a surface of Euler genus $g \geq 2$ is defined to be
$$H(g) = \frac{7 + \sqrt{1+24g}}{2}.$$
The function $H(g)$ appears in two related problems on embeddings of dense graphs. First, Heawood \cite{Heawood} showed that every $g$-embeddable graph, for $g \geq 2$, can be colored with $h(g) := \lfloor H(g) \rfloor$ colors.\footnote{In the literature, the Heawood number is usually defined to be $h(g)$, but our second problem involving the function $H(g)$ rounds this value up, instead.} Ringel, Youngs, and others \cite{Ringel-MapColor} showed that this bound is tight for every surface except the Klein bottle by finding an embedding of the complete graph on $h(g)$ vertices in that surface:

\begin{theorem}[see Ringel \cite{Ringel-MapColor}]
For $n \geq 3$, the orientable genus of $K_n$ is
$$\gamma(K_n) = \left\lceil \frac{(n-3)(n-4)}{12} \right\rceil.$$
\label{thm-mct-orient}
\end{theorem}
\begin{theorem}[see Ringel \cite{Ringel-MapColor}]
For $n \geq 5$, $n \neq 7$, the nonorientable genus of $K_n$ is
$$\tilde{\gamma}(K_n) = \left\lceil \frac{(n-3)(n-4)}{6} \right\rceil.$$
\label{thm-mct-nonorient}
\end{theorem}

A graph is said to be $h$-critical if it has chromatic number $h$ and every proper subgraph is $(h-1)$-colorable. As a consequence of a construction by Fisk \cite{Fisk} (see Corollary 8.4.13 in Mohar and Thomassen \cite{MoharThomassen}), there are infinitely many $5$-critical $g$-embeddable graphs, for all $g \geq 1$. In sharp contrast, Euler's formula implies that when $h \geq 8$, there are finitely many $h$-critical graphs embeddable in each surface: as the number of vertices increases, the average degree approaches $6$, but $h$-critical graphs have minimum degree at least $h-1 \geq 7$. Thomassen \cite{Thomassen-6Critical} further showed that there are finitely many $6$- and $7$-critical $g$-embeddable graphs, for each fixed $g$. These results lead to a natural problem:

\begin{question}
For each closed surface and for each $h \geq 6$, what are the $h$-critical graphs embeddable in that surface? 
\end{question}

For $h = 6$, these graphs are known for the sphere (none, by the five-color theorem \cite{Heawood}), projective plane \cite{AlbertsonHutchinson}, torus \cite{Thomassen-Torus}, and Klein bottle \cite{Klein1,Klein2}. Our focus is on the other end of the spectrum. A result of Dirac \cite{Dirac-MapColor1, Dirac-MapColor2}, with some missing cases supplied by Albertson and Hutchinson \cite{AlbertsonHutchinson}, complements Theorems \ref{thm-mct-orient} and \ref{thm-mct-nonorient} by showing that if a graph is $g$-embeddable and $h(g)$-critical, then it must be the complete graph $K_{h(g)}$.

\v{S}krekovski \cite{Skrekovski} introduced a problem generalizing Dirac's result: what are the $(h(g)-c)$-critical $g$-embeddable graphs, for some fixed integer $c \geq 1$? For sufficiently large $g$, \v{S}krekovski showed that all such graphs have at most $h(g) + 1$ vertices. In particular, when $c = 1$ and $g \geq 5$, $g \neq 6, 9$, \v{S}krekovski ruled out every other graph besides the complete graph $K_{h(g)-1}$ and the graph join $K_{h(g)-4} + C_5$, where $C_5$ is the 5-cycle. The former graph is $g$-embeddable simply because it is a subgraph of $K_{h(g)}$. We show the necessity of the latter graph, which we prefer to write as $K_{h(g)+1}-C_5$, for certain surfaces:

\begin{theorem}
For $n \geq 5$, the orientable genus of $K_n-C_5$ is
$$\gamma(K_n-C_5) = \left\lceil \frac{n^2-7n+2}{12} \right\rceil.$$
\label{thm-main-orient}
\end{theorem}
\begin{theorem}
For $n \geq 7$, the nonorientable genus of $K_n-C_5$ is
$$\tilde{\gamma}(K_n-C_5) = \left\lceil \frac{n^2-7n+2}{6} \right\rceil.$$
\label{thm-main-nonorient}
\end{theorem}

\v{S}krekovski \cite[Proposition 3.2]{Skrekovski} showed that these genus formulas imply the following:

\begin{corollary}
$K_{h(g)+1}-C_5$ is embeddable in the orientable surface of Euler genus $g \geq 2$ if and only if $g = 2(\gamma(K_{h(g)+1})-1)$ and $h(g) \not\equiv 2, 3, 6, 11 \pmod{12}$. 
\end{corollary}
\begin{corollary}
$K_{h(g)+1}-C_5$ is embeddable in the nonorientable surface of Euler genus $g \geq 2$ if and only if either
\begin{itemize}
\item $g = \tilde{\gamma}(K_{h(g)+1})-1$, or
\item $g = \tilde{\gamma}(K_{h(g)+1})-2$ and $h(g) \equiv 1 \pmod{3}$.
\end{itemize}
\end{corollary}

Theorems \ref{thm-main-orient} and \ref{thm-main-nonorient} are proven in Sections \ref{sec-orient} and across Sections \ref{sec-min} and \ref{sec-nonorient}, respectively. We also determine the $(h(g)-1)$-critical graphs for Euler genus $g = 6, 9$ in Section \ref{sec-skrek69}. In the following result, the graphs $H_7$ and $M_7$ are the two 4-critical graphs on 7 vertices, shown in Figure \ref{fig-4crit}:

\begin{figure}[htbp]
\centering
\includegraphics[scale=1]{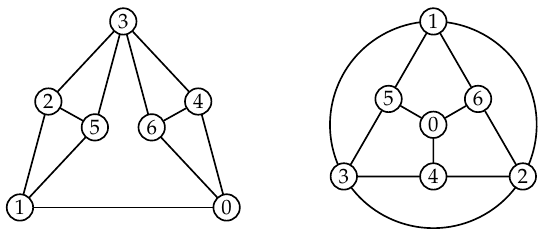}
\caption{The 4-critical graphs $H_7$ and $M_7$ on seven vertices.}
\label{fig-4crit}
\end{figure}

\begin{theorem}
The $8$-critical graphs embeddable in the surfaces of Euler genus 6 are $K_8$, $K_5 + C_5$, and $K_4 + H_7$. The $9$-critical graphs embeddable in the nonorientable surface of Euler genus 9 are $K_9$, $K_6 + C_5$, $K_5 + H_7$, and $K_5 + M_7$. 
\label{thm-69}
\end{theorem}

The Heawood number also appears in the \emph{minimum triangulations} problem: every simple triangulation in a surface of Euler genus $g$ has at least $\lceil H(g) \rceil$ vertices, and, except for some small cases, there are tight examples in the form of complete graphs with a small number of edges deleted. The main technique used in our proof of Theorem \ref{thm-main-nonorient} also gives a new and simpler construction for minimum triangulations for each nonorientable surface. 

Ma and Ren \cite{Ma-Cycle} showed that the orientable genus of $C_m+K_n$ is $\lceil(m-2)(n-2)/4\rceil = \gamma(K_{m,n})$ for $m \geq 6n-13$. Roughly speaking, when $m$ is large compared to $n$, one expects that there exists a special genus embedding of the complete bipartite graph $K_{m,n}$ where the remaining edges can be added as diagonals of the (non-triangular) faces. Ellingham, Stephens, and Zha \cite{Ellingham-Tripartite,EllinghamStephens} observe a similar phenomenon for other joins, like $E_m+K_n$ and $E_m + K_{n,p} = K_{m,n,p}$, where $E_m$ denotes the empty graph on $m$ vertices. Since we are dealing with a small, fixed value of $m$, our constructions are similar to those of genus embeddings of complete graphs, which utilize the theory of current graphs. 

\section{Background}

The notation $K_n-H$ denotes the graph formed by taking the complete graph $K_n$ and deleting the edges of a copy of $H$. We use $P_m$ to denote the path graph on $m$ edges. The disjoint union of two graphs $G$ and $H$ is written as $G \cup H$. 

The classification of closed surfaces states that the closed orientable surfaces are the $k$-holed tori, denoted $S_k$, for each $k \geq 0$, and the closed nonorientable surfaces are spheres with $\ell$ crosscaps, denoted $N_\ell$, for $\ell \geq 1$. The Euler characteristic of $S_k$ and $N_\ell$ are $2-2k$ and $2-\ell$, respectively, hence their Euler genus is $2k$ and $\ell$, respectively. In this work, all embeddings $\phi\colon G \to \Sigma$, where $\Sigma$ is one of the aforementioned surfaces, are cellular, i.e., the connected components of $\Sigma\setminus\phi(G)$, which we call \emph{faces}, are homeomorphic to open disks. Let $F$ denote the set of faces of $\phi$. Cellular embeddings satisfy the \emph{Euler polyhedral formula}
$$|V(G)|-|E(G)|+|F| = \chi = 2-g,$$
where $g$ is the Euler genus of $S$. 

Each face can be thought of as an $f$-sided polygon, and we describe a face using a cyclic sequence $[v_1, v_2, \dotsc, v_f]$ of the vertices incident with the face. We call $f$ the \emph{length} of the polygon. If $f = 3$, we say the face is \emph{triangular}, and if $f = 4$, we call the face \emph{quadrangular}. An embedding is said to be \emph{triangular} if every face is triangular. In any embedding of a simple graph on at least three vertices, each face has length at least 3, which implies a standard consequence of the Euler polyhedral formula:

\begin{proposition}
A simple connected graph $G$ on at least 3 vertices embedded in a surface of Euler genus $g$ has at most $3|V(G)|-6+3g$ edges, with equality if and only if the embedding is triangular. 
\label{prop-eulerupper}
\end{proposition}

The \emph{orientable genus} $\gamma(G)$ of a graph $G$ is the smallest integer $k$ where $G$ embeds in $S_k$, and the \emph{nonorientable genus} $\tilde{\gamma}(G)$ is defined analogously. By rearranging Proposition \ref{prop-eulerupper}, we obtain the so-called ``Euler lower bound'' on the two genus parameters:

\begin{corollary}
Let $G$ be a simple connected graph on at least three vertices. The orientable genus of $G$ is at least
$$\gamma(G) \geq \left\lceil \frac{|E(G)|-3|V(G)|+6}{6}\right\rceil.$$
The nonorientable genus of $G$ is at least
$$\tilde{\gamma}(G) \geq \left\lceil \frac{|E(G)|-3|V(G)|+6}{3}\right\rceil.$$
\label{cor-eulerlower}
\end{corollary}

We note that Theorems \ref{thm-mct-orient}, \ref{thm-mct-nonorient}, \ref{thm-main-orient}, and \ref{thm-main-nonorient} are the statements that the genera of the complete graphs (except $K_7$ in Theorem \ref{thm-mct-nonorient}) and $K_n-C_5$ match this lower bound. However, unlike the complete graphs, the graph $K_n - C_5$ never triangulates a surface, so we find genus embeddings by first constructing triangular embeddings of a supergraph, and then deleting the appropriate edges. 

\begin{proposition}
Let $G$ be a 6-edge-connected (resp. 3-edge-connected) graph with a triangular embedding in an orientable (resp. nonorientable) surface. Then, deleting up to five (resp. two) edges from this embedding yields an orientable (resp. nonorientable) genus embedding of the resulting graph.
\label{prop-del}
\end{proposition}
\begin{proof}
We prove the orientable case---the nonorientable case differs only in the denominator of the Euler lower bound. Since there exists a triangular embedding of $G = (V, E)$, then $(|E|-3|V|+6)/6$ is an integer. After deleting up to five edges, the graph is still connected, so the Euler lower bound of the resulting graph is at least
$$\left\lceil\frac{(|E|-5)-3|V|+6}{6}\right\rceil = \frac{|E|-3|V|+6}{6} + \left\lceil -\frac{5}{6} \right\rceil = \frac{|E|-3|V|+6}{6}.$$
\end{proof}

\v{S}krekovski \cite{Skrekovski} applied this result to orientable triangular embeddings of $K_n-K_2$, which exist for $n \equiv 2, 5 \pmod{12}$ (see \cite{Sun-Index2, Ringel-MapColor}), to obtain orientable genus embeddings of $K_n - C_5$. It can also be applied to orientable triangular embeddings of $K_n$, $n \equiv 0, 3, 4, 7 \pmod{12}$, as well:

\begin{proposition}[\v{S}krekovski \cite{Skrekovski}]
For $n \equiv 0, 2, 3, 4, 5, 7 \pmod{12}$, $n \geq 5$, there exists an orientable embedding of $K_n - C_5$ whose genus equals the Euler lower bound. 
\label{prop-skrek}
\end{proposition}

For the nonorientable genus of $K_n-C_5$ and the remaining residues in the orientable case, we find triangular embeddings of complete graphs with three or four edges missing. Let $P_m$ denote the path graph on $m$ edges. For three edges, we have a choice to construct an embedding of either $K_n - (K_2 \cup P_2)$ or $K_n - P_3$, but for four edges, the only option is $K_n-P_4$. 

In the nonorientable case, triangular embeddings of $K_n-P_3$ were found by Ringel \cite{Ringel-MinNon} for most residues $n$ modulo $12$ as part of a more general construction for minimum triangulations. The full result for minimum triangulations in all closed surfaces was proven in two papers by Jungerman and Ringel:

\begin{theorem}[Jungerman and Ringel \cite{Ringel-MinNon, JungermanRingel-Minimal}]
Besides the surfaces $S_2$, $N_2$, and $N_3$, the fewest number of vertices in a simple triangulation of a surface of Euler genus $g$ is $\lceil H(g) \rceil.$ For the remaining surfaces, there is a simple triangulation on $\lceil H(g) \rceil + 1$ vertices. 
\label{thm-mintri}
\end{theorem}

The exceptions to the above formula were demonstrated by Huneke \cite{Huneke}, Franklin \cite{Franklin-SixColor}, and Ringel \cite{Ringel-MinNon}, respectively. For all other cases, Ringel \cite{Ringel-MinNon} observed that the problem can be phrased in terms of finding triangular embeddings of graphs on specific numbers of vertices and edges. An \emph{$(n,t)$-triangulation} is a triangular embedding of a simple graph on $n$ vertices and $\binom{n}{2}-t$ edges. 

\begin{lemma}[Ringel \cite{Ringel-MinNon}]
To construct a minimum triangulation for all nonorientable surfaces of Euler genus at least 4, it suffices to find a nonorientable $(n,t)$-triangulation for all $n \geq 9$ and nonnegative $t \leq n-6$, where $t \equiv 1 \pmod{3}$ if $n \equiv 2 \pmod{3}$ and $t \equiv 0 \pmod{3}$, otherwise. 
\label{lem-nt}
\end{lemma}

Then, each residue $n$ modulo $12$ is handled separately, similar to the proofs of Theorems \ref{thm-mct-orient} and \ref{thm-mct-nonorient}. However, Ringel's original argument \cite{Ringel-MinNon} appears to be incomplete for $n \equiv 1 \pmod{12}$. For this residue, Ringel used an inductive construction (see Section 10.1 of Ringel \cite{Ringel-MapColor} for a current graph interpretation):

\begin{lemma}[Ringel \cite{Ringel-MinNon}]
For each $t \geq 2$, if there exist $(2t+1, r_1)$-, $(2t+1, r_2)$-, and $(2t+1, r_3)$-triangulations, where each graph has a vertex adjacent to all other vertices, then there exists a $(6t+1, r_1+r_2+r_3)$-triangulation. If any of the former three embeddings are nonorientable, then the latter embedding is nonorientable as well.
\end{lemma}

For the smallest case $t = 2$, there is a $(5,1)$-triangulation in the plane, and there does not exist a $(5,4)$-triangulation. Thus, this method can only construct an orientable $(13,3)$-triangulation and cannot construct a $(13,6)$-triangulation. The gap in this construction affects every order $n = 12s+1$, where $s$ is a power of $3$. Given the work on computationally enumerating small triangulations of surfaces (e.g., Ellingham and Stephens \cite{Ellingham-Neighborly} and Sulanke and Lutz \cite{SulankeLutz}), almost certainly the missing nonorientable triangulations have been discovered at some point, but the author is not aware of any explicit examples in the literature.

While we can cite Ringel's other results in our proofs, the author felt obligated to find a new method for the nonorientable part of Theorem \ref{thm-mintri}. Roughly speaking, for the other residues $n \not\equiv 1 \pmod{12}$, Ringel's proof proceeds in two steps:
\begin{itemize}
\item find an appropriate triangular embedding of a complete or near-complete graph on $n$ vertices, then
\item repeatedly modify the rotation system to obtain $(n,t)$-triangulations for larger values of $t$.
\end{itemize}
For the first step, Ringel reused embeddings from his earlier work, mostly from his proof of Theorem \ref{thm-mct-nonorient} \cite{Ringel-Non}. These were found using his ``leading permutation'' method, which apparently Ringel himself (see \cite[pg.9]{Ringel-MapColor}) considered to be very difficult. Youngs \cite{Youngs-Nonorientable} (with further improvements by Jungerman \cite{Jungerman-Case1} and Korzhik \cite{Korzhik-AnotherProof,Korzhik-SimpleProof}) simplified the proof of Theorem \ref{thm-mct-nonorient} using current graphs, but it is not clear whether the second step in Ringel's plan, which we call \emph{crosscap subtraction}, can be performed on all of these embeddings. In Section \ref{sec-min}, we provide a sufficient condition for repeated crosscap subtraction in embeddings derived from current graphs. Then, we combine old and new families of current graphs (including one for $n \equiv 1 \pmod{12}$) satisfying this condition to reprove the nonorientable part of Theorem \ref{thm-mintri}.

\section{Combinatorial embeddings and current graphs}

For more background on topological graph theory and the theory of current graphs, see Gross and Tucker \cite{GrossTucker} and Ringel \cite{Ringel-MapColor}. 

To specify cellular embeddings combinatorially, we start by distinguishing the ends of each edge in $E$. Each edge $e$ induces two arcs $e^+$ and $e^-$ pointing in opposite directions, and we denote the set of such arcs by $E^+$. Cellular embeddings in surfaces are specified by a (general) \emph{rotation system}, where each vertex is assigned a \emph{rotation}, a cyclic permutation of the arcs leaving the vertex, and the edges are assigned a \emph{signature} $\lambda\colon E \to \{-1, +1\}$. An edge with signature $+1$ is said to be \emph{normal} and an edge with signature $-1$ is said to be \emph{twisted}. A face-boundary walk has two \emph{behaviors}, a \emph{normal} behavior where the next arc after $u \to v$ is chosen to be the arc in the rotation of $v$ after $v \to u$, and an \emph{alternate} behavior where the next arc chosen is the one before $v \to u$. The face-boundary walk switches to the other behavior when it traverses a twisted edge. We say that an edge is \emph{bidirectional} if the two times a face-boundary walk traverses the edge is in opposite directions, and \emph{unidirectional}, otherwise. The embedding is nonorientable if and only if there is a cycle in the graph with an odd number of twisted edges. In particular, when all of the edges in a rotation system are normal, the embedding is orientable. 

Given a triangular embedding of a simple graph, if there are two triangular faces $[a, b, c]$ and $[b, a, d]$ sharing the edge $(a, b)$, then the \emph{edge flip} $(a, b) \to (c, d)$ deletes $(a, b)$ and adds $(c, d)$ in the resulting quadrangular face. If $c$ and $d$ are already adjacent, then we delete the original edge connecting those two vertices, which triggers another edge flip. A \emph{sequence of edge flips}, denoted by 
$$(u_1, v_1) \to (u_2, v_2) \to \dotsc \to (u_{i-1}, v_{i-1}) \to (u_i, v_i),$$
deletes the edges $(u_1, v_1), \dotsc, (u_{i-1}, v_{i-1})$ and adds $(u_2, v_2), \dotsc, (u_i, v_i)$ in the resulting quadrangular faces, respectively.

A \emph{current graph} is an embedded graph $\phi\colon G \to \Sigma$, given as a rotation system, with an arc-labeling $\alpha\colon E^+ \to \Gamma$ with elements from a finite abelian group $\Gamma$. In this work, $\Gamma$, the \emph{current group}, is always a cyclic group $\Z_m$. For each edge $e$, the arc-labeling satisfies $\alpha(e^+) = -\lambda(e)\alpha(e^-)$, where $\lambda$ is the edge signature. The \emph{index} of the current graph is the number of faces in the embedding, and it is required that the index divides the order of the current group $m$. Face boundaries are called \emph{circuits}, and they are labeled $[0], [1], \dotsc, [k-1]$, where $k$ is the index. Given a circuit that traverses the cyclic sequence of arcs $(e^\pm_1, e^\pm_2, \dotsc)$, the \emph{log} of the circuit replaces $e^\pm_i$ with $\alpha(e^\pm_i)$ if the circuit is in normal behavior, and $-\alpha(e^\pm_i)$, otherwise. 

The excess of a vertex is the sum of the currents entering that vertex, and if the excess is 0, we say that vertex satisfies \emph{Kirchhoff's current law}. Most vertices are of degree 3 and satisfy Kirchhoff's current law, which cause them to generate triangular faces, but we also make use of special vertices, called \emph{vortices}, that do not satisfy Kirchhoff's current law, and hence generate longer faces. Given a current graph of index $k$, each vortex has degree $k$, is incident with each of the $k$ circuits, and has excess that either generates (1) the subgroup $k \Z_m$, or (2) the subgroup $(2k) \Z_m$. We call such a vertex a \emph{vortex of type (V1) or (V2)}, respectively. We take the convention that vortex labels are incorporated into logs as a reminder of where the nontriangular faces will be. 

Except for the current graphs used in Lemma \ref{lem-induction}, the current graphs in this work satisfy the following standard properties:

\begin{itemize}
\item[(C1)] Each vertex has degree 1 or 3.
\item[(C2)] The log of each circuit contains each nonzero element of the current group exactly once. 
\item[(C3)] Unlabeled vertices of degree 3 satisfy Kirchhoff's current law.
\item[(C4)] The excess of each unlabeled vertex of degree 1 has either order 2 or 3 in the current group.
\item[(C5)] Labeled vertices are vortices of type (V1) or (V2). 
\item[(C6)] For each edge $e$, if circuit $[i]$ passes through $e^+$ in, say, normal behavior, and $[j]$ passes through $e^-$, then $j - i \equiv \alpha(e^+) \pmod{k}$. 
\end{itemize}

By property (C2), the derived graph will be the complete graph with vertex set $\mathbb{Z}_m$. To obtain the rotation system of the derived embedding of the current graph, we first temporarily ignore vortex labels in each of the logs. The rotation at vertex $v \in \mathbb{Z}_m$ is generated by the ``additivity rule'': take the log of circuit $[v \bmod{k}]$ and add $v$ to every element. An edge in the derived embedding is twisted if and only if the corresponding edge in the current graph is unidirectional. Vortices of type (V1) generate Hamiltonian $m$-sided faces, which we subdivide with a vertex of the same name. Vortices of type (V2) generate two $(m/2)$-sided faces. If the vortex label, is, say, $y$, we subdivide the face incident with vertex $0$ with a vertex labeled $y_0$, and the other face with a vertex labeled $y_1$. However, we note that $y_1$ is not necessarily adjacent to vertex $1$ when the index of the current graph is greater than 1. We call vertices that come from the current group \emph{numbered}, and the subdivision vertices from vortices \emph{lettered}.

We describe infinite families of current graphs using standard building blocks. For current graphs of index 1 and 2, this is a simple ladder, as seen in Figure \ref{fig-sampleladder}(a), where the ``rungs,'' i.e., the vertical arcs, alternate in direction and have currents that form an arithmetic sequence. To determine the rotations for the missing vertices, we refer to the two rungs separated by the ellipses. If the ``upper'' vertices have different colors, then they form the checkerboard pattern shown in Figure \ref{fig-sampleladder}(a). This pattern will only be used for the index 2 current graphs in Figure \ref{fig-orient11}. If the upper vertices have the same color, as will be the case in all other infinite families of index 1 and 2 in this work, then every rung has the same colors for the upper and lower vertices. 

Infinite families of index 3 current graphs have the more complicated ladder structure in Figure \ref{fig-sampleladder}(b), where every other rung is what Ringel and Youngs called \emph{globular}: there are two additional vertices connected by two parallel edges. The currents on those arcs match the horizontal arcs above and below each such arc. 

\begin{figure}[htbp]
\centering
    \begin{subfigure}[b]{0.99\textwidth}
        \centering
        \includegraphics[scale=0.95]{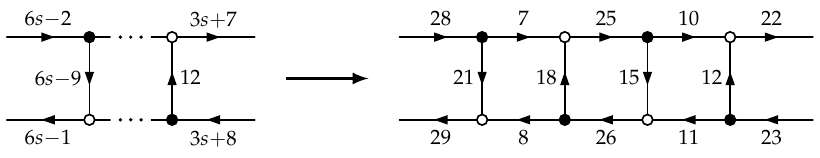}
        \caption{}
    \end{subfigure}
    \begin{subfigure}[b]{0.99\textwidth}
        \centering
        \includegraphics[scale=0.89]{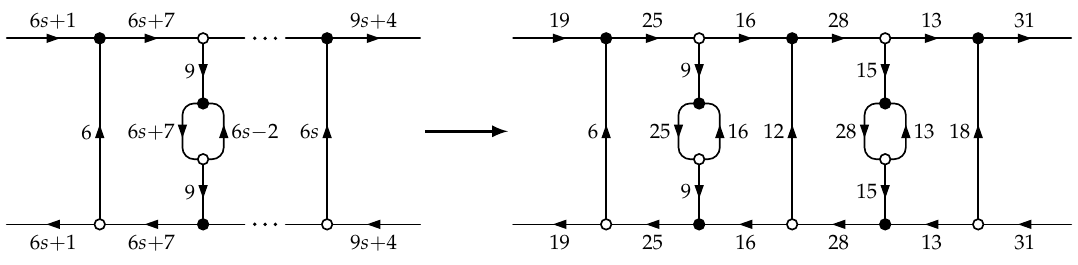}
        \caption{}
    \end{subfigure}
\caption{A ladder for index 2 (a) and index 3 (b) and their specifications for $s = 5$ and $s = 3$, respectively.}
\label{fig-sampleladder}
\end{figure}

\section{Orientable constructions} \label{sec-orient}

For brevity, all embeddings in this section are orientable. Gross \cite{Gross-Case6} observed that one can flip edges in a triangular embedding of $K_n-K_3$ to produce $(n,3)$-triangulations with different underlying graphs. We only need one such edge flip:

\begin{construction}
Given a triangular embedding of $K_n-K_3$, if $x$ and $y$ are two nonadjacent vertices and there is a vertex $u$ whose rotation is of the form
$$\begin{array}{rccccccccccccccccccccccccccccccc}
u. & \dotsc & x & v & y & \dotsc, \\
\end{array}$$
for some other vertex $v$, then flip the edge $(u, v) \to (x,y)$ to obtain a triangular embedding of $K_n-(K_2 \cup P_2)$. 
\label{con-k2p2}
\end{construction}

\begin{lemma} 
For $n \equiv 1, 6, 9, 10 \pmod{12}$, $n \geq 6$, $n \neq 9$, there exists a triangular embedding of $K_n - (K_2 \cup P_2)$ or $K_n - P_3$.
\label{lem-16910}
\end{lemma}

\begin{proof}
For embeddings derived from current graphs, Construction \ref{con-k2p2} is possible when the current graph has three vortices of type (V1), and two of them are adjacent, such as in the families for $n = 12s+1$, $s \geq 2$ \cite{Sun-Kainen}, $n = 12s+6$, $s \geq 2$ \cite{Sun-Minimum}, and $n = 12s+9$, $s \geq 1$ \cite{Sun-Kainen}. In the triangular embeddings of $K_{13}-K_3$ and $K_{18}-K_3$ of Jungerman \cite{Jungerman-K18}, we may flip the edges $(2, 6) \to (11, 12)$ and $(8, 10) \to (16, 17)$, respectively. Finally, $K_6-P_3$ is planar.

We are left with the last residue, $n = 12s+10$. A triangular embedding of $K_{10} - P_3$ is presented in Sun \cite{Sun-FaceDist}. The author \cite{Sun-Kainen} (in particular, refer to Figures 18(b) and 22(c)) has constructed a drawing of $K_n$ in the surface of genus $\gamma(K_n)-1$ with three crossings. One can choose an edge to delete from each crossing to obtain the desired triangular embedding. For $s = 1$, we get a triangular embedding of $K_n - (K_2 \cup P_2)$ that is missing the edges $(d, b)$, $(d, 15)$, and $(2, 12)$.\footnote{The current arXiv version of \cite{Sun-Kainen} contains an error in Figure 18(b): the edge $(1, 8)$ crosses $(d, 15)$, not $(z, 15)$.} For $s \geq 2$, we get a triangular embedding of $K_n-P_3$ that is missing the edges $(a, f)$, $(f, e)$, and $(e, c)$. 
\end{proof}

For the exceptional case $n = 9$, Huneke \cite{Huneke} proved that there is no $(9,3)$-triangulation. However, it is possible to embed some of the graphs with four edges missing:

\begin{proposition}
The genus of $K_9-C_5$ is $2$. 
\label{prop-k9}
\end{proposition}
\begin{proof}
The Euler lower bound for $K_9-C_5$ is 2. Now consider the following triangular rotation system:
$$\begin{array}{rccccccccccccccccccccccccccccccccc}
0. & x & 2 & 7 & 5 & 3 & 8 & 6 & 4 \\
1. & & 3 & 5 & 8 & 7 & 4 & 6 \\
2. & x & 5 & 6 & 8 & 4 & 7 & 0 \\
3. & & 0 & 5 & 1 & 6 & 7 & 8 \\
4. & x & 0 & 6 & 1 & 7 & 2 & 8 & 5 \\
5. & x & 4 & 8 & 1 & 3 & 0 & 7 & 6 & 2 \\
6. & & 0 & 8 & 2 & 5 & 7 & 3 & 1 & 4 \\
7. & & 0 & 2 & 4 & 1 & 8 & 3 & 6 & 5 \\
8. & & 0 & 3 & 7 & 1 & 5 & 4 & 2 & 6 \\
x. & & 0 & 4 & 5 & 2 \\
\end{array}$$
One can check that the genus of this embedding is 2. After deleting vertex $x$, the missing edges form the path $[0, 1, 2, 3, 4]$. Deleting the edge $(0, 4)$ yields the desired embedding. 
\end{proof}

\begin{remark}
Kenta Noguchi (personal communication) has classified all 9-vertex graphs embeddable in $S_2$ without computer search. In particular, the only graphs with $\binom{9}{2}-4$ edges that do not embed are $K_9 -4K_2$, $K_9 - (K_3 \cup K_2)$ and $K_9-(K_1+(K_1 \cup K_2))$. A proof of the nonembeddability of the first graph appears in Su, Noguchi, and Zhou \cite{Su-Rings}. By Proposition \ref{prop-del}, all graphs formed by deleting five edges from $K_9$ embed in $S_2$, since each graph has a spanning supergraph on $\binom{9}{2}-4$ edges that is not one of the aforementioned graphs.
\end{remark}

For the remaining residues, we apply techniques familiar to Jungerman and Ringel \cite{JungermanRingel-Minimal} for constructing $(n,4)$-triangulations. Then, we apply additional edge flips to ensure that the missing edges form a path. 

\begin{construction}
Suppose we have a current graph of index $1$ or $3$ and current group $\mathbb{Z}_{12s+6}$ with one vortex of type (V1) labeled $x$, one vortex of type (V2) labeled $y$, and no other vortices. Suppose further that the log of circuit $[0]$ is of the form
$$\begin{array}{rccccccccccccccccc}
\dotsc & y & a & 6s{+}3 & \dotsc,
\end{array}$$
where $a$ is a numbered vertex. Delete the edges $(a, 0)$, $(0, 6s+3)$, and $(6s+3, (6s+3)+a)$. In the resulting face, add an edge between $y_0$ and $y_1$ and contract it to form a vertex $y$.
\label{con-con8}
\end{construction}

\begin{figure}[htbp]
\centering
\includegraphics[scale=1]{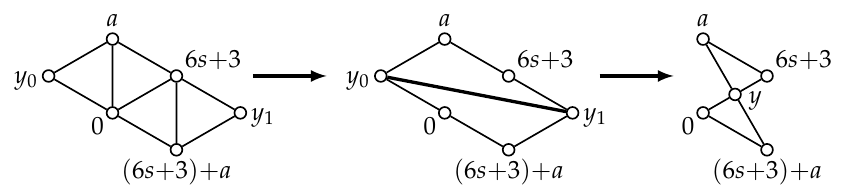}
\caption{Amalgamating the two vertices induced by a (V2) vortex.}
\label{fig-contract8}
\end{figure} 

Construction \ref{con-con8} is visualized in Figure \ref{fig-contract8}. The missing edges are $(x,y)$ and the deleted path of length 3.

\begin{lemma}
For $n \equiv 8 \pmod{12}$, $n \geq 20$, there exists a triangular embedding of $K_n-P_4$. 
\label{lem-orient8}
\end{lemma}
\begin{proof}
For $n = 12s+8$, $s \geq 1$, consider the family of current graphs in Figure \ref{fig-orient8}. The vortices $x$ and $y$ are of type (V1) and (V2), respectively. After applying Construction \ref{con-con8} with $a = 3s+1$, we have a triangular embedding of $K_{12s+8}$ with the edges $(3s+1, 0)$, $(0, 6s+3)$, $(6s+3, 9s+4)$, and $(x,y)$ missing. Flipping the edge $(3s+1, 9s+5) \to (x, y)$ yields the desired embedding.

\begin{figure}[htbp]
\centering
\includegraphics[scale=1]{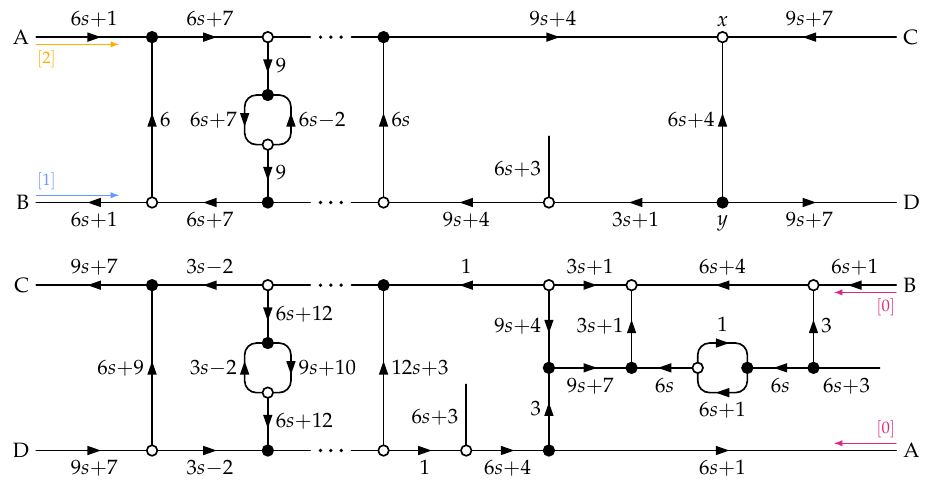}
\caption{Index 3 current graphs with group $\mathbb{Z}_{12s+6}$, $s \geq 1$.}
\label{fig-orient8}
\end{figure}
\end{proof}

\begin{lemma}
For $n \equiv 11 \pmod{12}$, $n \geq 11$, there exists a triangular embedding of $K_n-P_4$. 
\label{lem-orient11}
\end{lemma}
\begin{proof}
For $n = 11$, we use the following triangular rotation system, where the missing edges form the path $[0, 1, 2, 3, 4]$:
$$\begin{array}{rccccccccccccccccccccccccccccccccc}
0. & 2 & 5 & 4 & 10 & 9 & 3 & 6 & 8 & 7 \\
1. & 3 & 9 & 8 & 5 & 10 & 4 & 6 & 7 \\
2. & 0 & 7 & 10 & 8 & 6 & 4 & 9 & 5 \\
3. & 0 & 9 & 1 & 7 & 5 & 8 & 10 & 6 \\
4. & 0 & 5 & 7 & 8 & 9 & 2 & 6 & 1 & 10 \\
5. & 0 & 2 & 9 & 6 & 10 & 1 & 8 & 3 & 7 & 4 \\
6. & 0 & 3 & 10 & 5 & 9 & 7 & 1 & 4 & 2 & 8 \\
7. & 0 & 8 & 4 & 5 & 3 & 1 & 6 & 9 & 10 & 2 \\
8. & 0 & 6 & 2 & 10 & 3 & 5 & 1 & 9 & 4 & 7 \\
9. & 0 & 10 & 7 & 6 & 5 & 2 & 4 & 8 & 1 & 3 \\
10. & 0 & 4 & 1 & 5 & 6 & 3 & 8 & 2 & 7 & 9
\end{array}$$

For $n = 12s+11$, $s \geq 1$, we start with the triangular embeddings of $K_n - K_5$ in Figure \ref{fig-orient11} for $s \geq 2$, and Figure 25 of Sun \cite{Sun-Minimum} for $s = 1$. The vortices $a, b, c, d, e$ in these current graphs are all of type (V1). In each rotation system, there is a vertex $u$ whose rotation is of the form 

$$\begin{array}{rccccccccccccccccc}
u. & \dotsc & a & v_1 & b & \dotsc & c & v_2 & d & v_3 & e & \dotsc \\
\end{array}$$

Successively flip the edges $(u, v_1) \to (a, b)$, $(u, v_2) \to (c, d)$, $(u, v_3) \to (d, e)$, $(u, d) \to (c, e)$. Then, by connecting the faces $[u, b, a]$ and $[c, e, d]$ with a handle like in Figure \ref{fig-handle11}, one can add back $(u, d)$ and all of the missing edges between the lettered vertices except $(b, c)$. At this point, the missing edges are $(b, c)$, $(u, v_1)$, $(u, v_2)$, and $(u, v_3)$. 

For $s = 1$, the numbered vertices are $(u, v_1, v_2, v_3) = (1, 3, 9, 8)$. We apply the sequences of edge flips $(3, 16) \to (11, 17) \to (1, 8)$ and $(2, 16) \to (b, c)$. Now, the missing edges form the path $[2, 16, 3, 1, 9]$. For the larger cases, $(u, v_1, v_2, v_3)$ are $(0, 12s+3, 6s+5, 6s+9)$ for even $s \geq 2$ and $(0, 6s+5, 9s+4, 3s+2)$ for odd $s \geq 3$. The sequences of edge flips starting by deleting $(c, v_3)$ are
$$(c, 6s+9) \to (4, 6s+6) \to (0, 6s+5) = (u, v_2)$$
for even $s$, and 
$$(c, 3s+2) \to (6s+4, 6s+6) \to (0, 6s+5) = (u, v_1)$$
for odd $s$. The embeddings now have the missing path $[b, c, 6s+9, 0, 12s+3]$ or $[b, c, 3s+2, 0, 9s+4]$, respectively.

\begin{figure}[htbp]
\centering
    \begin{subfigure}[b]{0.99\textwidth}
        \centering
        \includegraphics[scale=0.95]{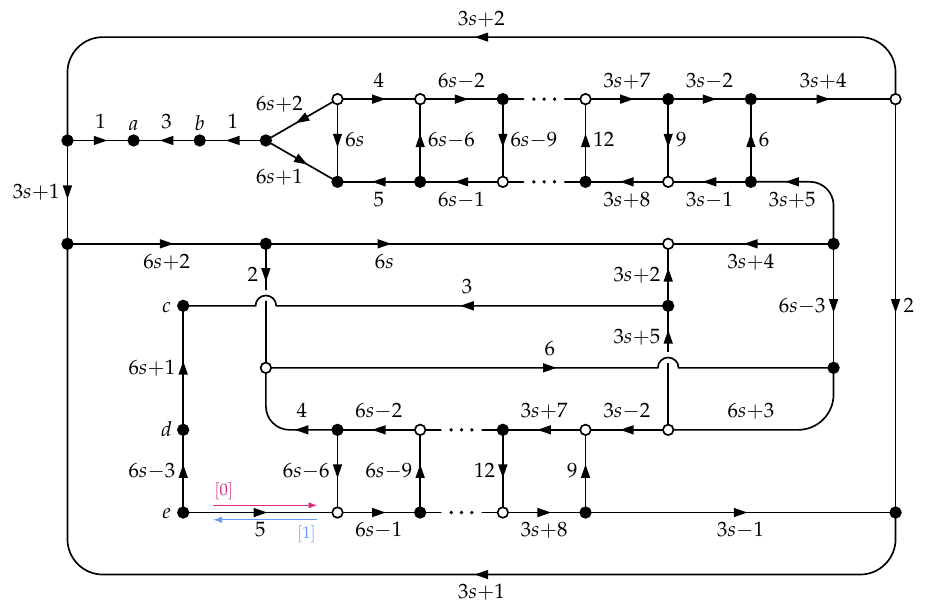}
        \caption{}
    \end{subfigure}
    \begin{subfigure}[b]{0.99\textwidth}
        \centering
        \includegraphics[scale=0.95]{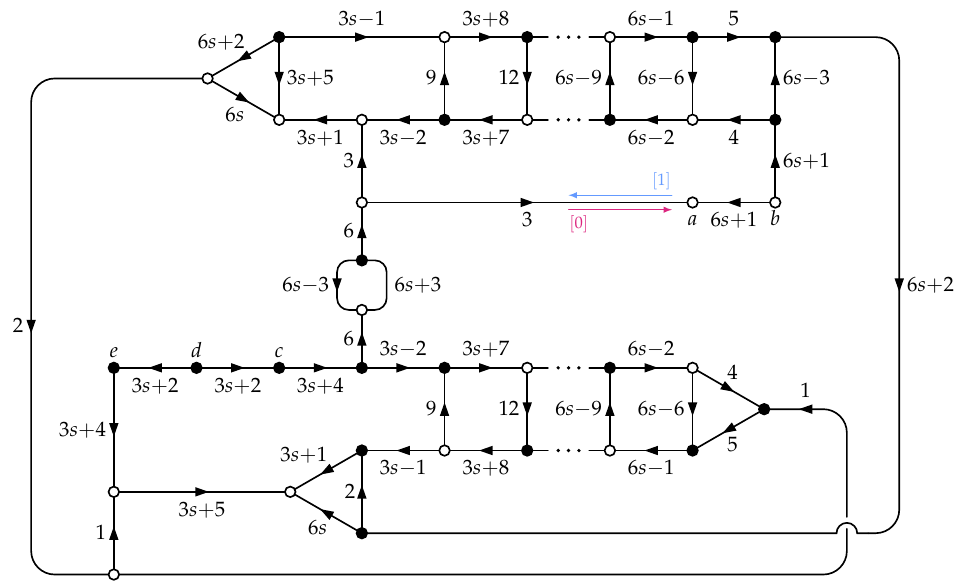}
        \caption{}
    \end{subfigure}
\caption{Families of index 2 current graphs with current group $\mathbb{Z}_{12s+6}$, for even $s \geq 2$ (a) and odd $s \geq 3$ (b).}
\label{fig-orient11}
\end{figure}

\begin{figure}[htbp]
\centering
    \begin{subfigure}[b]{0.75\textwidth}
        \centering
        \includegraphics[scale=1]{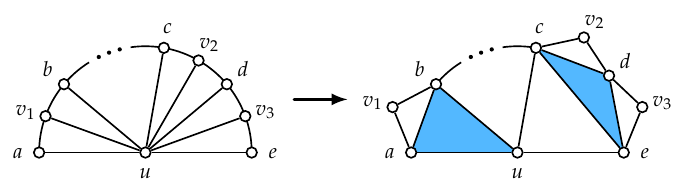}
        \caption{}
    \end{subfigure}
    \begin{subfigure}[b]{0.24\textwidth}
        \centering
        \includegraphics[scale=1]{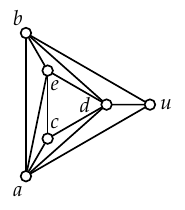}
        \caption{}
    \end{subfigure}
\caption{Adding most of the edges between the lettered vertices with one handle.}
\label{fig-handle11}
\end{figure}
\end{proof}

The final case is $n = 8$. As proven by Duke and Haggard \cite{DukeHaggard}, $K_8-P_4$ does not embed in the torus, but $K_8-C_5$ does.

\begin{proposition}[Duke and Haggard \cite{DukeHaggard}]
The genus of $K_8-C_5$ is 1. 
\label{prop-k8}
\end{proposition}

By combining Propositions \ref{prop-skrek}, \ref{prop-k9}, and \ref{prop-k8} and Lemmas \ref{lem-16910}, \ref{lem-orient8}, and \ref{lem-orient11}, we obtain the desired genus formula:

\begin{theorem}
For all $n \geq 5$, the orientable genus of $K_n - C_5$ matches its Euler lower bound. 
\label{thm-orient}
\end{theorem}

\section{A detour for nonorientable minimum triangulations}\label{sec-min}

Recall that Ringel \cite{Ringel-MinNon} found nonorientable minimum triangulations via a process we called crosscap subtraction. More specifically, the edges of paths of length 3 are deleted, and the rotation system is modified to keep the embedding triangular:

\begin{construction}[Ringel \cite{Ringel-MinNon}]
Suppose the rotation system of a (necessarily nonorientable) triangular embedding of $G$ is of the form
$$\begin{array}{rlllllllllllllllllllllllllllll}
b. & \dotsc & p & d & c & a & \dotsc \\
\end{array}$$
and
$$\begin{array}{rlllllllllllllllllllllllllllll}
p. & \dotsc & a & c & \dotsc & b & d & \dotsc \\
\end{array}$$
for some edge signature. Then, delete the edges of the path $[a, c, b, d]$, reverse the subsequence $c \dotsc b$ in the rotation of $p$, and switch the signature of each edge $(p, c) \dotsc, (p, b)$. The resulting embedding is a (possibly orientable) triangular embedding of $G - \{(a, c), (c, b), (b, d)\}$. 
\label{con-subtract}
\end{construction}

One can show that the original embedding has to be nonorientable and that the final embedding is triangular, most easily by observing that it is the reverse of the operation shown in Figure \ref{fig-crossadd}. The edges $(p, c), \dotsc, (p, b)$ are passed through a crosscap, merging the faces $[p, b, a]$ and $[p, d, c]$, which is then triangulated with the edges of the path in Figure \ref{fig-crossadd}(b). We call this configuration of four faces a \emph{subtractible crosscap} with vertices $(b, p, d, c, a)$. For the minimum triangulations problem, we need to subtract multiple crosscaps to obtain $(n,t)$-triangulations for large values of $t$. To this end, we say that two or more subtractible crosscaps are \emph{independent} if none of the crosscaps share any faces. Construction \ref{con-subtract} affects the four faces in a subtractible crosscap, but since it leaves the other faces intact, it can be successively applied to a set of pairwise independent subtractible crosscaps. 

\begin{figure}[htbp]
\centering
    \begin{subfigure}[b]{0.69\textwidth}
        \centering
        \includegraphics[scale=1]{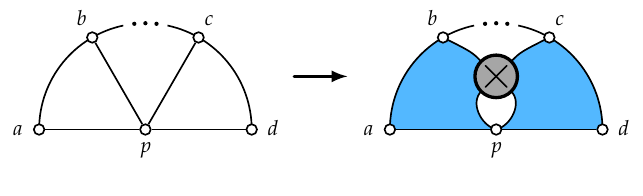}
        \caption{}
    \end{subfigure}
    \begin{subfigure}[b]{0.30\textwidth}
        \centering
        \includegraphics[scale=1]{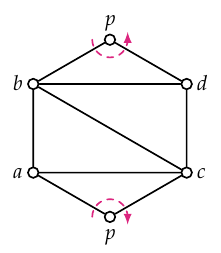}
        \caption{}
    \end{subfigure}
\caption{The reverse of crosscap subtraction.}
\label{fig-crossadd}
\end{figure}

\begin{remark}
At the time, Ringel did not have the concept of an edge signature, instead relying on two equivalent properties of rotations called ``Rule R'' and ``Rule $\Delta$'' (see Ringel \cite[p.131]{Ringel-MapColor}). In modern terms, if the rotation system satisfies these rules, then there is an edge signature which makes the embedding triangular. 
\end{remark}

In the orientable minimum triangulations problem, Jungerman and Ringel \cite{JungermanRingel-Minimal} describe a certain ladder substructure in a current graph that allows for ``handle subtraction.'' We describe an analogue for crosscap subtraction, which we call the \emph{broken 2-ladder}, shown in Figure \ref{fig-broken}. Crucially, the same circuit traverses the rungs of this ladder (the edges with currents $\pm \gamma$ and $\pm 2\gamma$), and the rungs are unidirectional. We note that the edges of a broken 2-ladder may appear to form a small ``local'' substructure in the graph, but the conditions on the circuits are ``global'' in that the behavior of the circuit as it passes through the broken 2-ladder depends on the rest of the current graph. We also observe that, since we can flip vertices, the signatures of the edges in a broken 2-ladder do not necessarily follow Figure \ref{fig-broken}. The presentation shown here indicates how the broken 2-ladder was discovered: as a fragment of some of the current graphs found in Ringel \cite{Ringel-MapColor}.

\begin{figure}[htbp]
\centering
\includegraphics[scale=1.2]{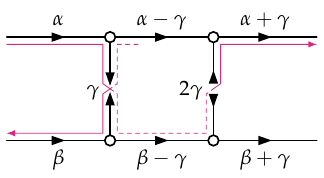}
\caption{A fragment of a current graph that induces subtractible crosscaps.}
\label{fig-broken}
\end{figure}

\begin{lemma}
Suppose we have a current graph of index $k$ and current group $\mathbb{Z}_{ks}$ that generates a triangular embedding. If it has a broken 2-ladder where $k$ divides the currents on the rungs of the broken 2-ladder, then the embedding has $s$ independent subtractible crosscaps.  
\label{lem-broken}
\end{lemma}
\begin{proof}
Suppose, without loss of generality, that circuit $[0]$ passes through the rungs of the broken 2-ladder. Its log is of the form
$$\begin{array}{rlllllllllllllllllllllllllllll}
\dotsc & \alpha & \gamma & (\gamma-\beta) & 2\gamma & (\alpha+\gamma) & \dotsc & -\gamma & -\beta & \dotsc
\end{array}$$
Since the same circuit traverses each rung twice, by property (C6), $\gamma$ is a multiple of $k$. Thus, the rotation at $\gamma$, which is also generated by this log, is of the form
$$\begin{array}{rlllllllllllllllllllllllllllll}
\dotsc & (\alpha+\gamma) & 2\gamma & \dotsc & 0 & (\gamma-\beta) & \dotsc
\end{array}$$
This is the form in Construction \ref{con-subtract}, where $b = 0$, $p = \gamma$, $d = \gamma-\beta$, $c = 2\gamma$, and $a = \alpha+\gamma$. By the additivity rule, we find $s$ subtractible crosscaps in total. Since each of the four faces modified by Construction \ref{con-subtract} is induced by a distinct vertex in the broken 2-ladder, and each such vertex generates $s$ triangular faces, none of the subtractible crosscaps share any faces and hence are all pairwise independent.
\end{proof}

In the interest of minimizing the number of distinct families of current graphs used in this work, for some residues, we leverage the inductive construction described in Section 10.2 of Ringel \cite{Ringel-MapColor}. For the family of current graphs used in that construction, Figure \ref{fig-induction} provides an alternative current assignment that the author believes is simpler to verify, but splits into two cases depending on parity. 

\begin{figure}[htbp]
\centering
    \begin{subfigure}[b]{0.99\textwidth}
        \centering
        \includegraphics[scale=1]{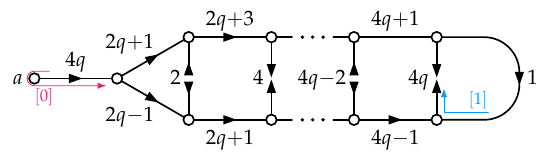}
        \caption{}
    \end{subfigure}
    \begin{subfigure}[b]{0.99\textwidth}
        \centering
        \includegraphics[scale=1]{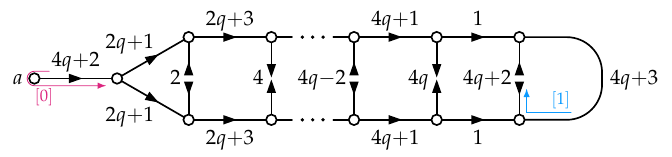}
        \caption{}
    \end{subfigure}
\caption{Current graphs with group $\mathbb{Z}_{8q+2}$ and $\mathbb{Z}_{8q+6}$ used in an inductive construction.}
\label{fig-induction}
\end{figure}

\begin{lemma}
For $r \geq 2$, if there exists a $(2r+2, t)$-triangulation with one vertex adjacent to every other vertex, then there exists a nonorientable $(4r+2, t+3i)$-triangulation, for any $i = 0, 1, \dotsc, 2r$. 
\label{lem-induction}
\end{lemma}
\begin{proof}
Consider the two families of index 2 current graphs with current group $\mathbb{Z}_{4r+2}$ in Figure \ref{fig-induction}, defined for all $r \geq 1$, where $r = 2q$ or $r = 2q+1$. The log of circuit $[1]$ contains every nonzero element of $\mathbb{Z}_{4r+2}$ exactly once. On the other hand, the log of circuit $[0]$ contains each odd element exactly once and does not contain any even elements except $\pm 2r$. Furthermore, the vortex labeled $a$ generates a $(2r+1)$-sided face incident with all of the even elements. Relabel the vertices of the $(2r+2, t)$-triangulation so that the vertex adjacent to every other vertex is called $a'$ and its rotation is of the form
$$\begin{array}{rlllllllllllllllllllllllllllll}
a'. & \dotsc & 0 & 2r & 2(2r) & 3(2r) & \dotsc, \\
\end{array}$$
matching the ordering of the face induced by vortex $a$. Delete $a'$ and its incident edges and faces and the face induced by vortex $a$ and glue the two surfaces along the resulting boundaries, identifying the pairs of vertices with the same name. This yields a $(4r+2, t)$-triangulation. 

When $r \geq 2$, there is a broken 2-ladder with currents $\pm 2$ and $\pm 4$ on the rungs, so we may subtract up to $2r+1$ crosscaps using Lemma \ref{lem-broken}. Since none of these crosscaps involve the vortex $a$, this does not interfere with the aforementioned surgery operation. This process yields $(4r+2, t+3i)$-triangulations, for each $i = 0, 1, \dotsc, 2r+1$. To avoid having to explicitly check for nonorientability, we use a shortcut of Ringel \cite{Ringel-MinNon}: the presence of a subtractible crosscap immediately implies nonorientability, so the triangulations for $i = 0, \dotsc, 2r$ (but not necessarily $i = 2r+1$) are nonorientable.
\end{proof} 

\begin{corollary}
For $n \equiv 6, 10 \pmod{12}$, $n \geq 10$, there exist nonorientable $(n, t)$-triangulations for $t \leq n-6$, $t \equiv 0 \pmod{3}$. For $n \equiv 2 \pmod{12}$, $n \geq 26$, there exist nonorientable $(n, t)$-triangulations for $t \leq n-6$, $t \equiv 1 \pmod{3}$. 
\label{cor-2610}
\end{corollary}
\begin{proof}
We apply Lemma \ref{lem-induction} with $t = 0$ or $1$ and different values of $i$. Write $n$ as $4r+2$ for some positive integer $r$. If $r \neq 0 \pmod{3}$, then by Theorem \ref{thm-mct-nonorient}, there exists a triangular embedding of $K_{2r+2}$. Otherwise, we need a triangular embedding of $K_{2r+2}-K_2$. If $2r+2 \equiv 2 \pmod{12}$, then we may use the orientable triangular embeddings of $K_{12s+2}-K_2$, $s \geq 1$, mentioned earlier in Proposition \ref{prop-skrek}. If $2r+2 \equiv 8 \pmod{12}$, then we use the nonorientable triangular embeddings of $K_{12s+8}-K_2$, $s \geq 1$, presented in Korzhik \cite{Korzhik-Case8} or Sun \cite{Sun-Index2}. 
\end{proof}

For some of the other residues, we gather known families of current graphs with broken 2-ladders:

\begin{itemize}
\item For $n \equiv 0, 4 \pmod{12}$, $n \geq 16$, the families ``$\Gamma(12s+12)$'' and ``$\Gamma(12s+4)$,'' respectively, in Figure 9 of Korzhik \cite{Korzhik-AnotherProof}. 
\item For $n \equiv 5, 11 \pmod{12}$, $n \geq 17$, the families in Figures 8.18--21 of Ringel \cite{Ringel-MapColor}.
\item For $n \equiv 3, 9 \pmod{12}$, $n \geq 15$, the index 3 family in Figure 9.16 of Ringel \cite{Ringel-MapColor} (see the Remark on page 158 for $n \equiv 3 \pmod{12}$), formed by combining three copies of the aforementioned current graphs for $n \equiv 5, 11 \pmod{12}$. 
\end{itemize}

In each of these families, there are at most two vortices of type (V1), meaning that the order of the current group is always at least $n - 2$. The index of each current graph is at most $3$, so by Lemma \ref{lem-broken}, we are able to subtract at least $n - 2$ edges. By Ringel's shortcut for determining nonorientability, we obtain:

\begin{lemma}
For $n \equiv 0, 3, 4, 9 \pmod{12}$, $n \geq 15$, there exist nonorientable $(n, t)$-triangulations for all $t \leq n-6$, $t \equiv 0 \pmod{3}$. For $n \equiv 5, 11 \pmod{12}$, $n \geq 11$, there exist nonorientable $(n, t)$-triangulations for all $t \leq n-6$, $t \equiv 1 \pmod{3}$. 
\label{lem-prevknown}
\end{lemma}

It remains to solve the cases $n \equiv 1, 7, 8 \pmod{12}$ and some small sporadic cases. 

\begin{proposition}
For $n = 9, 12$, there exist nonorientable $(n, t)$-triangulations, for $t \equiv 0 \pmod{3}$ and $t \leq n-6$. For $n = 11, 14$, there exist nonorientable $(n, t)$-triangulations, for $t \equiv 1 \pmod{3}$ and $t \leq n-6$.
\label{prop-smallmin}
\end{proposition}
\begin{proof}
For $n = 9, 11, 12$, we consider the current graphs belonging to the aforementioned families, shown in Figure \ref{fig-smallcurrent}(a--c). We also utilize the index 2 current graph in Figure \ref{fig-smallcurrent}(d), which generates a nonorientable triangular embedding of $K_{14}-K_2$. None of these current graphs have broken 2-ladders, but we are still able to successively subtract the following crosscaps:
 
\begin{itemize}
\item For $n = 9, 11$, $(b, p, d, c, a) = (0, 3, 5, 6, 2)$ and $(1, 4, 6, 7, 3)$.
\item For $n = 12$, $(b, p, d, c, a) = (0, 10, x, 1, 4)$, $(2, 1, x, 3, 6)$, and $(4, 3, x, 5, 8)$. 
\item For $n = 14$, $(b, p, d, c, a) = (0, 10, 8, 9, y)$, $(2, 0, 10, 11, y)$, and $(4, 2, 0, 1, y)$.
\end{itemize} 
\end{proof}

\begin{remark}
The $(9,6)$-triangulation is a minimum triangulation for the surface $N_3$, one of the exceptions in Theorem \ref{thm-mintri}. It is nonorientable because the Euler genus is odd.
\end{remark}

We caution that attempts to use the additivity rule to generate multiple subtractible crosscaps for these current graphs, as done in Lemma \ref{lem-broken}, can result in crosscaps that share faces. For example, in the $n = 12$ case, we used subtractible crosscaps of the form $(b, p, d, c, a) = (i, 10+i, x, 1+i, 4+i)$. The independent crosscaps we obtained were from setting $i = 0, 2, 4$. However, the two crosscaps where $i = 0,1$ overlap at the face $[0, 1, x]$. This possibility arises because vortex $x$ does not generate triangular faces, but a Hamiltonian face that we subdivided later. For $n = 11$, the crosscaps of the form $(b, p, d, c, a) = (i, 3+i, 5+i, 6+i, 2+i)$ conflict when $i = 0$ and $3$. In this case, the four triangular faces in each subtractible crosscap are generated from just two vertices in the current graph. 

\begin{figure}[htbp]
\centering
    \begin{subfigure}[b]{0.55\textwidth}
        \centering
        \includegraphics[scale=1]{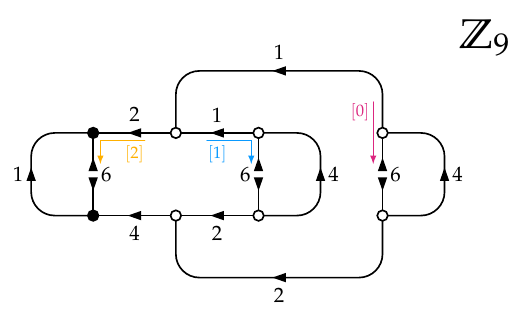}
        \caption{}
    \end{subfigure}
    \begin{subfigure}[b]{0.28\textwidth}
        \centering
        \includegraphics[scale=1]{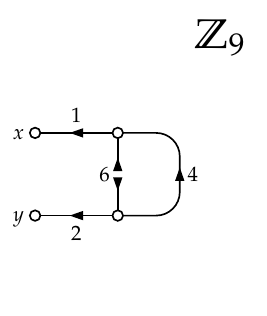}
        \caption{}
    \end{subfigure}
    \begin{subfigure}[b]{0.34\textwidth}
        \centering
        \includegraphics[scale=1]{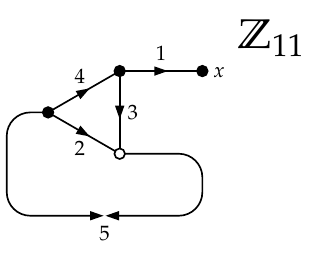}
        \caption{}
    \end{subfigure}
    \begin{subfigure}[b]{0.59\textwidth}
        \centering
        \includegraphics[scale=1]{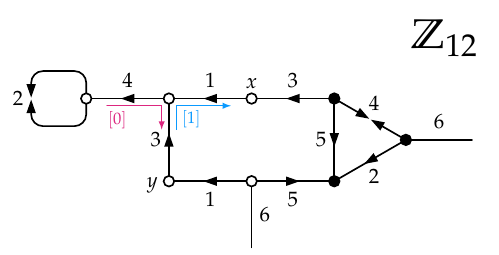}
        \caption{}
    \end{subfigure}
\caption{Small current graphs with subtractible crosscaps.}
\label{fig-smallcurrent}
\end{figure}

\begin{lemma}
For $n \equiv 1, 7 \pmod{12}$, $n \geq 13$, there exist nonorientable $(n, t)$-triangulations for all $t \leq n-6$, $t \equiv 0 \pmod{3}$.
\label{lem-min17}
\end{lemma}
\begin{proof}
Consider the families of current graphs in Figures \ref{fig-non1} ($n \equiv 1$) and \ref{fig-non7} ($n \equiv 7$). Each current graph has a broken 2-ladder except for the graph in Figure \ref{fig-non1}(a). For this case, we use the \emph{ad hoc} subtractible crosscaps $(b, p, d, c, a) = (i, i+12, i+11, i+4, i+7)$. The crosscaps for $i = 0, 1, 2$ are independent.
\end{proof}

\begin{figure}[htbp]
\centering
    \begin{subfigure}[b]{0.99\textwidth}
        \centering
        \includegraphics[scale=1]{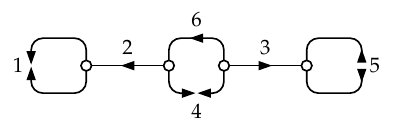}
        \caption{}
    \end{subfigure}
    \begin{subfigure}[b]{0.99\textwidth}
        \centering
        \includegraphics[scale=1]{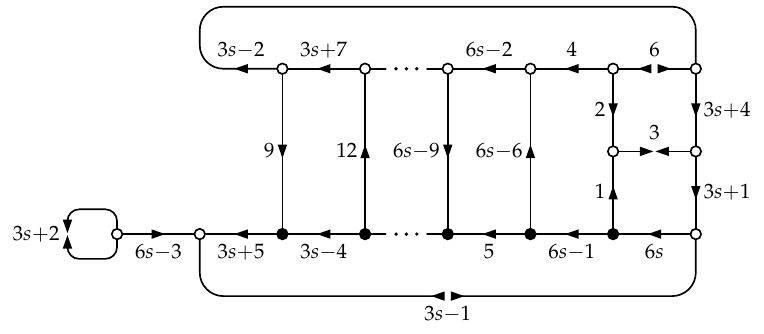}
        \caption{}
    \end{subfigure}
\caption{Index 1 current graphs with group $\Z_{12s+1}$ for $s = 1$ (a) and $s \geq 2$ (b).}
\label{fig-non1}
\end{figure}

\begin{figure}[htbp]
\centering
    \begin{subfigure}[b]{0.99\textwidth}
        \centering
        \includegraphics[scale=1]{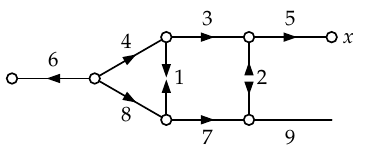}
        \caption{}
    \end{subfigure}
    \begin{subfigure}[b]{0.99\textwidth}
        \centering
        \includegraphics[scale=1]{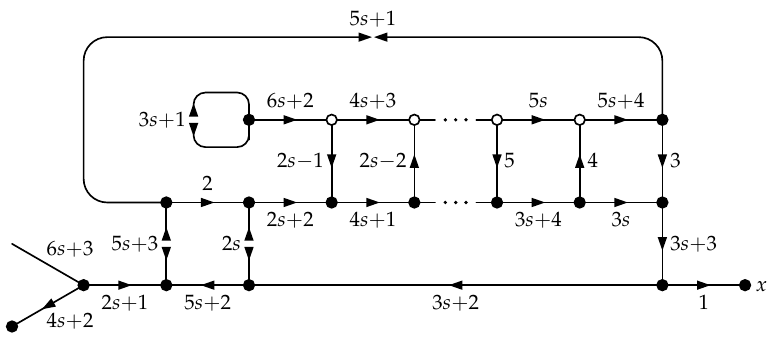}
        \caption{}
    \end{subfigure}    
\caption{Index 1 current graphs with group $\Z_{12s+6}$ for $s = 1$ (a) and $s \geq 2$ (b).}
\label{fig-non7}
\end{figure}

\begin{lemma}
For $n \equiv 8 \pmod{12}$, $n \geq 20$, there exist nonorientable $(n, t)$-triangulations for all $t \leq n-6$, $t \equiv 1 \pmod{3}$.
\label{lem-min8}
\end{lemma}
\begin{proof}
We follow Ringel's approach \cite{Ringel-Non} of solving the $t = 1$ and the $t \geq 4$ cases separately. For $t = 1$, these are the triangular embeddings of $K_{12s+8}-K_2$ mentioned in Corollary \ref{cor-2610}. For $t \geq 4$, consider the family of current graphs in Figure \ref{fig-non8}. We apply Construction \ref{con-con8} with $a = 6s+4$. For $s \geq 2$, there is a broken 2-ladder. For $s = 1$, we use the subtractible crosscaps $(b, p, d, c, a) = (i, i+9, i+1, i+13, i+15)$. These crosscaps are induced by four distinct vertices, so they are all independent. However, Construction \ref{con-con8} destroyed the crosscaps corresponding to $i = 0$ and $9$, since it deleted the edge $(0, 9)$. We instead subtract the four crosscaps corresponding to $i = 1, \dotsc, 4$. 
\end{proof}

\begin{figure}[htbp]
\centering
\includegraphics[scale=1]{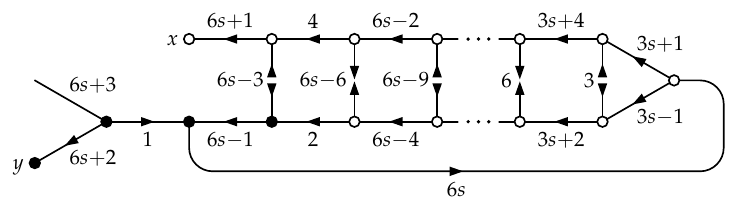}
\caption{Index 1 current graphs with group $\Z_{12s+6}$ for $s \geq 1$.}
\label{fig-non8}
\end{figure}

Corollary \ref{cor-2610}, Lemmas \ref{lem-prevknown}, \ref{lem-min17}, and \ref{lem-min8}, and Proposition \ref{prop-smallmin} cover all the requisite cases:

\begin{theorem}
There exists a nonorientable $(n,t)$-triangulation for all $n \geq 9$ and nonnegative $t \leq n-6$, where $t \equiv 1 \pmod{3}$ if $n \equiv 2 \pmod{3}$ and $t \equiv 0 \pmod{3}$, otherwise. 
\end{theorem}

Hence, by Lemma \ref{lem-nt}, the nonorientable part of Theorem \ref{thm-mintri} is true for Euler genus at least 4. Returning to our original problem, whenever we constructed an $(n,3)$-triangulation, it was after applying crosscap subtraction once to a triangular embedding of $K_n$:

\begin{corollary}
For $n \not\equiv 2 \pmod{3}$, $n \geq 9$, there exists a nonorientable triangular embedding of $K_n-P_3$. 
\label{cor-non-p3}
\end{corollary}

\section{The remaining nonorientable cases for $K_n-C_5$} \label{sec-nonorient}

For $n \equiv 2 \pmod{3}$, the $(n, 4)$-triangulations constructed in the previous section are of the graph $K_n - (K_2 \cup P_3)$. Like in the orientable case, additional edge flips are needed to ensure that the missing edges form a single path.

\begin{lemma}
For $n = 12s + 11$, $s \geq 0$, there exists a nonorientable triangular embedding of $K_n-P_4$. 
\label{lem-skrek11}
\end{lemma}
\begin{proof}
For $s = 0$, we use the following triangular rotation system:

$$\begin{array}{rccccccccccccccccccccccccccccccccc}
0. & 2 & 7 & 10 & 3 & 5 & 6 & 8 & 9 & 4 \\
1. & 4 & \TT{7} & \TT{3} & 9 & 6 & 10 & 5 & 8 \\
2. & 0 & 4 & 6 & 5 & 10 & 9 & 8 & 7 \\
3. & 0 & 10 & 8 & 6 & 7 & \TT{1} & \TT{9} & 5 \\
4. & 0 & 9 & \TT{5} & \TT{7} & 1 & 8 & 10 & 6 & 2 \\
5. & 0 & 3 & \TT{9} & \TT{4} & 7 & 8 & 1 & 10 & 2 & 6 \\
6. & 0 & 5 & 2 & 4 & 10 & 1 & 9 & 7 & 3 & 8 \\
7. & 0 & 2 & 8 & 5 & \TT{4} & \TT{1} & 3 & 6 & 9 & 10 \\
8. & 0 & 6 & 3 & 10 & 4 & 1 & 5 & 7 & 2 & 9 \\
9. & 0 & 8 & 2 & 10 & 7 & 6 & 1 & \TT{3} & \TT{5} & 4 \\
10. & 0 & 7 & 9 & 2 & 5 & 1 & 6 & 4 & 8 & 3 \\
\end{array}$$

In each of our nonorientable rotation systems, the red underlined numbers indicate twisted edges. Nonorientability can be checked by observing that there is a vertex adjacent to every other vertex and not incident with any twisted edge. 

\begin{figure}[htbp]
\centering
\includegraphics[scale=1]{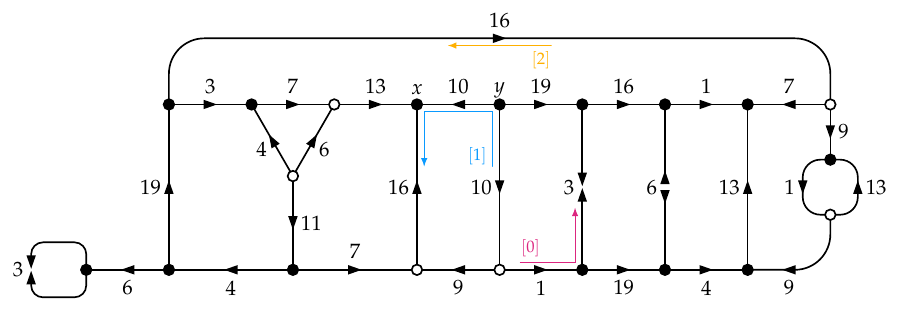}
\caption{An index 3 current graph with current group $\mathbb{Z}_{21}$.}
\label{fig-non11-s1}
\end{figure}

For $s = 1$, we refer to the current graph in Figure \ref{fig-non11-s1} containing a broken 2-ladder. After subtracting the crosscap with $(b, p, d, c, a) = (0, 3, 5, 6, 4)$, we flip $(4, 14) \to (x,y)$ so that the missing edges form the path $[14, 4, 6, 0, 5]$. 

For $s \geq 2$, we can reuse the construction from Lemma \ref{lem-orient11}. Instead of connecting the two faces $[0, b, a]$ and $[c, e, d]$ by a handle, we connect them by a \emph{twisted} handle: by reversing the orientation of the face $[0, b, a]$, the roles of vertices $a$ and $b$ are switched so that the missing edge is $(a,c)$ instead of $(b,c)$. The edge flips that comprise the rest of the construction are unchanged. 
\end{proof}

\begin{figure}[htbp]
\centering
\includegraphics[scale=1]{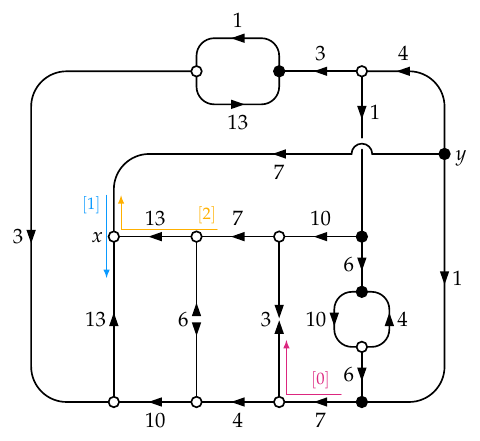}
\caption{An index 3 current graph with current group $\mathbb{Z}_{15}$.}
\label{fig-non5-s1}
\end{figure}

\begin{lemma}
For $n = 12s + 5$, $s \geq 1$, there exists a nonorientable triangular embedding of $K_n-P_4$. 
\label{lem-skrek5}
\end{lemma}
\begin{proof}
Consider the current graphs in Figures \ref{fig-non5-s1} and \ref{fig-skrek5}. We apply the same approach as the $s = 1$ case in Lemma \ref{lem-skrek11}. For $s = 1$ (resp.\ $s \geq 2$), after subtracting the crosscap $(b, p, d, c, a) = (0, 3, 8, 6, 10)$ (resp.\ $(b, p, d, c, a) = (0, 3, 6s+5, 6, 6s+7)$) using the broken 2-ladder, we flip the edge $(10, 2) \to (x, y)$ (resp.\ $(6s+7, 6s+8) \to (x, y)$) to obtain the missing path $[2, 10, 6, 0, 8]$ (resp. $[6s+8, 6s+7, 6, 0, 6s+5]$).
\end{proof}

\begin{figure}[htbp]
\centering
\includegraphics[scale=1]{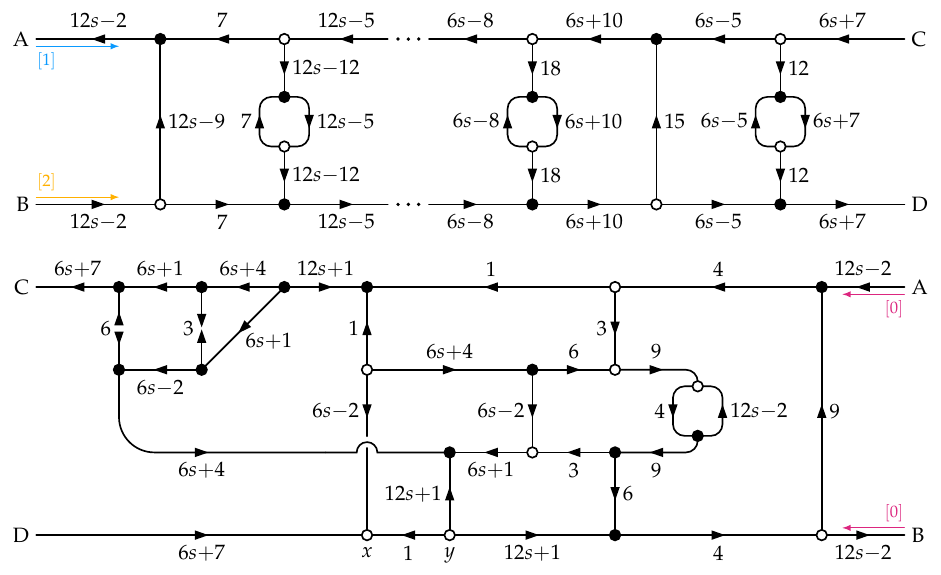}
\caption{Index 3 current graphs with group $\Z_{12s+3}$ for $s \geq 2$.}
\label{fig-skrek5}
\end{figure}

\begin{lemma}
For $n = 12s + 8$, $s \geq 0$, there exists a nonorientable triangular embedding of $K_n-P_4$. 
\label{lem-skrek8}
\end{lemma}
\begin{proof}
$K_8-P_4$ has a triangular embedding in the Klein bottle (see, e.g., the embeddings ``Kh2'' and ``Kh5'' in Lawrencenko and Negami \cite{LawrencenkoNegami}).

For $s = 1$, we use the index 3 current graph in Figure \ref{fig-skrek8-s1}. Like its orientable counterpart in Lemma \ref{lem-orient8}, we apply Construction \ref{con-con8} with $a = 7$, and then flip the edge $(7, 8) \to (x, y)$ to obtain the path $[8, 7, 0, 9, 16]$. 

\begin{figure}[htbp]
\centering
\includegraphics[scale=1]{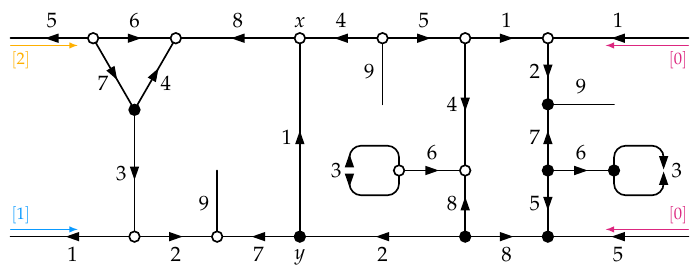}
\caption{Index 3 current graph with group $\Z_{18}$.}
\label{fig-skrek8-s1}
\end{figure}

For $s \geq 2$, we use the index 1 current graphs in Figure \ref{fig-skrek8}. For $s = 2$ (resp.\ $s \geq 3$), apply Construction \ref{con-con8} with $a = 28$ (resp. $a = 4$) and flip the edges $(28, x) \to (9, 17) \to (6, 18) \to (13, 15)$ (resp.\ $(4, x) \to (3, 5) \to (6s+10, 12s+4) \to (6s+3, 6s+7)$). In both cases, the missing path is $[6s+3, 0, a, x, y]$.  
\end{proof}

\begin{figure}[htbp]
\centering
    \begin{subfigure}[b]{0.99\textwidth}
        \centering
        \includegraphics[scale=1]{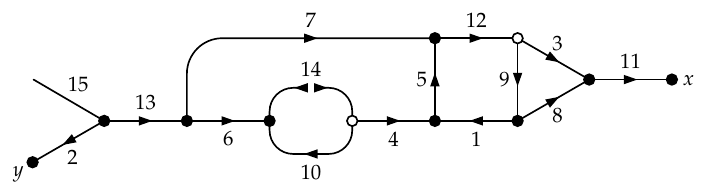}
        \caption{}
    \end{subfigure}
    \begin{subfigure}[b]{0.99\textwidth}
        \centering
        \includegraphics[scale=1]{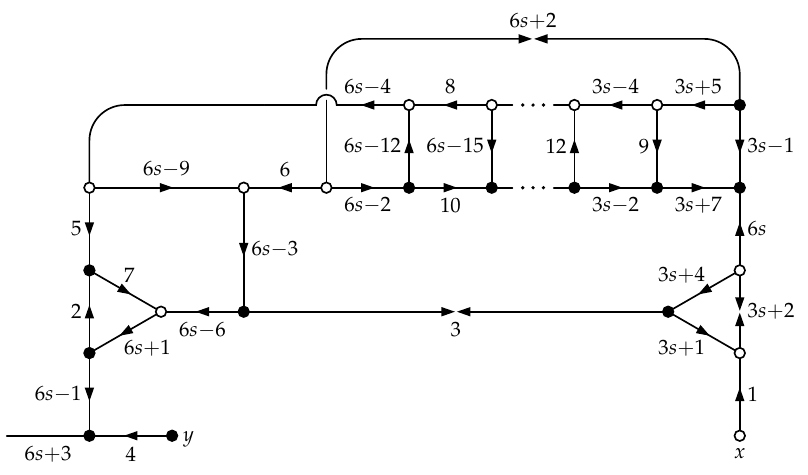}
        \caption{}
    \end{subfigure}
\caption{Index 1 current graphs with group $\Z_{12s+6}$ for $s = 2$ (a) and $s \geq 3$ (b).}
\label{fig-skrek8}
\end{figure}

\begin{remark}
The $(8,4)$-triangulation above is a minimum triangulation for the surface $N_2$, the other nonorientable exception in Theorem \ref{thm-mintri}.
\end{remark}

\begin{lemma}
For $n = 12s + 2$, $s \geq 1$, there exists a nonorientable triangular embedding of $K_n-P_4$. 
\label{lem-skrek2}
\end{lemma}

\begin{proof}
We induct using the surgical construction in Lemma \ref{lem-induction} with $t = 4$ and $r = 3s$. If $2r+2 \equiv 8 \pmod{12}$ (which includes the base case $s = 1$), we use the embeddings from Lemma \ref{lem-skrek8}. Otherwise, $2r+2 \equiv 2 \pmod{12}$ and existence follows by induction. 
\end{proof}

\begin{remark}
We note that, unlike in Corollary \ref{cor-2610}, the application of Lemma \ref{lem-induction} covers the $s = 1$ case because of the existence of an $(8, 4)$-triangulation. Thus, Lemma \ref{lem-induction} can also construct a nonorientable $(14, 7)$-triangulation. However, we still needed a separate construction for the $(14, 1)$-triangulation, since there are no $(8, 1)$-triangulations. 
\end{remark}

We are left with one final case that was not involved in the minimum triangulations problem:

\begin{proposition}
The nonorientable genus of $K_7-C_5$ is $1$.
\label{prop-non7}
\end{proposition}
\begin{proof}
By Theorem \ref{thm-mct-nonorient}, there is a triangular embedding of $K_6$ in $N_1$. Delete any edge and subdivide the resulting quadrangular face with a new vertex. The resulting graph is $K_7 - (K_2 \cup P_2)$. 
\end{proof}

By combining Corollary \ref{cor-non-p3}, Lemmas \ref{lem-skrek11}--\ref{lem-skrek2}, and Proposition \ref{prop-non7}, we obtain the full nonorientable genus formula: 

\begin{theorem}
For all $n \geq 6$, the nonorientable genus of $K_n - C_5$ matches its Euler lower bound. 
\label{thm-nonorient}
\end{theorem}

\section{Euler genus 6 and 9}\label{sec-skrek69}

Since $h(6) = 9$ and $h(9) = 10$, the goal of this section is to characterize the $8$- and $9$-critical graphs in the surfaces of Euler genus 6 and 9, respectively. \v{S}krekovski's calculation does not cover these cases, and indeed there are other $(h(g)-1)$-critical graphs embeddable in these surfaces. For the remainder of this section, when we speak of $h$-critical graphs, $h$ is at least $4$, unless otherwise specified. Label the consecutive vertices of the 5-cycle $0, 1, 2, 3, 4$. Then define $C_5(a_0, a_1, a_2, a_3, a_4)$ to be the graph where vertex $i$ of the 5-cycle is replaced with a clique on $a_i$ vertices, and each edge $(a_i, a_{(i+1)\bmod 5})$ of the 5-cycle is replaced with the edges of the complete bipartite graph $K_{a_i, a_{(i+1) \bmod 5}}$ between the cliques. For an example of such a graph, see Figure \ref{fig-efamily}. 

\begin{figure}[htbp]
\centering
\includegraphics[scale=1]{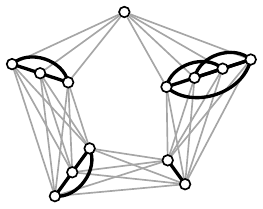}
\caption{The graph $C_5(1, 3, 3, 2, 4) \in \mathcal{E}_7$.}
\label{fig-efamily}
\end{figure}

Let $\mathcal{D}_h$ denote the family of graphs of the form $C_5(1, a, a', 1, h-2)$, where $a+a'=h-1$. Dirac \cite{Dirac-Critical} proved that $h$-critical graphs $G = (V, E)$ besides the complete graph $K_h$ satisfy $2|E| \geq (h-1)|V|+h-3$, sharpening the minimum-degree bound for $h$-critical graphs. The graphs in the family $\mathcal{D}_h$ show that this bound is tight. Gallai \cite{Gallai-Critical2} used this same family in an extension of Dirac's result, giving optimal bounds on $|E|$ for all $|V| \leq 2h-1$:

\begin{theorem}[Gallai \cite{Gallai-Critical2}]
If $G$ is an $h$-critical graph with at most $2h-2$ vertices, then $G$ is the join of two smaller critical graphs. 
\label{thm-join}
\end{theorem}

\begin{theorem}[Gallai \cite{Gallai-Critical2}]
For all $p = 2, \dotsc, h-1$, if $G = (V, E)$ is a $h$-critical graph with $|V| = h+p$ vertices, then $2|E| \geq (h-1)|V|+p(h-p)-2$, with equality if and only if $G$ is the join of $K_{h-p-1}$ and a graph from $\mathcal{D}_{p+1}$.
\label{thm-gallai}
\end{theorem}

For larger graphs, Kostochka and Stiebitz \cite{KostochkaStiebitz-Excess} gave a tighter bound than Dirac's. Define the family $\mathcal{E}_h$ to be the graphs $C_5(1, a, a', b', b)$, where $a+a' = h-1$, $b+b' = h-1$, and $a'+b' \leq h-1$. 

\begin{theorem}[Kostochka and Stiebitz \cite{KostochkaStiebitz-Excess}]
If $G = (V, E)$ is an $h$-critical graph and $G$ is not the complete graph $K_h$ or a member of $\mathcal{E}_h$, then $2|E| \geq (h-1)|V|+2(h-3)$.
\label{thm-ks}
\end{theorem}

By setting, say, $b' = 1$ in the definition of $\mathcal{E}_h$, we obtain $\mathcal{D}_h$ as a subfamily. As a consequence of Theorem \ref{thm-gallai} (or by a direct calculation), the graphs in $\mathcal{D}_h$ have the fewest number of edges out of all graphs in $\mathcal{E}_h$:

\begin{proposition}
The number of edges in a graph from $\mathcal{D}_h$ is $h^2-h-1$. All other graphs in $\mathcal{E}_h$ have strictly more edges.
\label{prop-edgecount}
\end{proposition}

For graphs of the form $C_5(a_0, a_1, a_2, a_3, a_4)$, if any $a_i$ is 1, we call the corresponding vertex \emph{lonely}. In a triangular embedding, the triangular faces around a vertex induce a Hamiltonian cycle on its neighbors. Thus, being \emph{locally Hamiltonian}, i.e., where the neighborhood of each vertex is Hamiltonian, is a necessary condition for the existence of a triangular embedding. To make a graph from $\mathcal{E}_h$ locally Hamiltonian, one would need to add at least two edges between neighbors of a lonely vertex. Since graphs in $\mathcal{D}_h$ have two lonely vertices, one would need to add at least four edges. 

Recall that there are two $4$-critical graphs on 7 vertices, $H_7$ and $M_7$ in Figure \ref{fig-4crit}. We note that $H_7$, the Haj\'os join of two copies of $K_4$, is the one graph in $\mathcal{D}_4$, and $M_7$ has one more edge than $H_7$. These graphs form the basis of a general pattern: 

\begin{proposition}[e.g., Exercises 4.12 and 4.13, Stiebitz, Schweser, Toft \cite{Stiebitz-Brooks}]
The $h$-critical graphs on $h+3$ vertices are $K_{h-4} + H_7$ and $K_{h-4} + M_7$. 
\label{prop-kp3}
\end{proposition}
\begin{proof}
We induct on $h$. The base case, $h = 4$, was established by Toft \cite{Toft-Critical} (note that $K_0 + G = G$).  For the inductive step, by Theorem \ref{thm-join}, any $h$-critical graph on $h+3$ vertices, for $h \geq 5$, is the join of two graphs. Since there are no $h$-critical graphs on $h+1$ vertices and for $h < 4$, there are no $h$-critical graphs on $h+3$ vertices, one of those two graphs is the complete graph on $h_1$ vertices, and the other is a $h_2$-critical graph on $h_2+3$ vertices, where $h_1 + h_2 = h$ and $4 \leq h_2 < h$. By the inductive hypothesis, the graph must be $K_{h_1} + (K_{h_2-4} + G) = K_{h-4} + G$, where $G$ is $H_7$ or $M_7$. 
\end{proof}

In each of the forthcoming rotation systems, we label the vertices of the complete graph with letters, and the vertices of $H_7$ and $M_7$ with numbers $0, \dotsc, 6$, as in Figure \ref{fig-4crit}. 

\begin{theorem}
The $8$-critical graphs embeddable in the surfaces of Euler genus 6 are $K_8$, $K_5 + C_5$, and $K_4 + H_7$. 
\label{thm-euler6}
\end{theorem}
\begin{proof}
$K_8$ embeds in these surfaces as a subgraph of $K_9$. For $|V| < 15$, one can check that the Gallai bound (Theorem \ref{thm-gallai}) combined with the Euler bound (Proposition \ref{prop-eulerupper}) $2|E| \leq 6|V|-12+12g$ rules out $|V| = 12, 13$. $K_{10}-C_5$ is the unique $8$-critical graph on 10 vertices, which embeds in these surfaces by Theorems \ref{thm-orient} and \ref{thm-nonorient}. For $|V| = 11$, the Gallai bound matches the Euler bound exactly, so the only possibility is $K_4+H_7$ by Proposition \ref{prop-kp3}. This graph embeds in the orientable surface $S_3$:

$$\begin{array}{rccccccccccccccccccccccccccccccc}
a. & b & d & 6 & 4 & c & 5 & 3 & 2 & 1 & 0 \\
b. & a & 0 & 4 & 6 & 3 & 5 & 2 & c & 1 & d \\
c. & a & 4 & 3 & 6 & 0 & 1 & b & 2 & d & 5 \\
d. & a & b & 1 & 5 & c & 2 & 3 & 4 & 0 & 6 \\
0. & a & 1 & c & 6 & d & 4 & b \\
1. & a & 2 & 5 & d & b & c & 0 \\
2. & a & 3 & d & c & b & 5 & 1 \\
3. & a & 5 & b & 6 & c & 4 & d & 2 \\
4. & a & 6 & b & 0 & d & 3 & c \\
5. & a & c & d & 1 & 2 & b & 3 \\
6. & a & d & 0 & c & 3 & b & 4
\end{array}$$

and the nonorientable surface $N_6$:

$$\begin{array}{rccccccccccccccccccccccccccccccccc}
a. & b & 3 & 4 & 6 & d & 0 & c & 5 & 1 & 2 \\
b. & a & 2 & d & \TT{4} & \TT{c} & 6 & 0 & 1 & 5 & 3 \\
c. & a & 0 & 4 & \TT{b} & \TT{6} & 3 & 2 & 1 & d & 5 \\
d. & a & 6 & \TT{3} & \TT{4} & b & 2 & 5 & c & 1 & 0 \\
0. & a & d & 1 & b & 6 & 4 & c \\
1. & a & 5 & b & 0 & d & c & 2 \\
2. & a & 1 & c & 3 & 5 & d & b \\
3. & a & b & 5 & 2 & c & \TT{6} & \TT{d} & 4 \\
4. & a & 3 & \TT{d} & \TT{b} & c & 0 & 6 \\
5. & a & c & d & 2 & 3 & b & 1 \\
6. & a & 4 & 0 & b & \TT{c} & \TT{3} & d
\end{array}$$

For $|V| = 14$, the Gallai and Euler bounds are equal again. The embedding must be triangular, and by Theorem \ref{thm-gallai}, the graph is the join of $K_1$ with a graph from $\mathcal{D}_7$. However, in this graph, the neighborhood of a lonely vertex is not locally Hamiltonian: the subgraph induced on its neighbors has the $K_1$ as a cutvertex. Thus, it cannot have a triangular embedding. 

For $|V| \geq 15$, the Kostochka-Stiebitz bound (Theorem \ref{thm-ks}) exceeds the Euler bound, so the remaining graphs to check are those belonging to $\mathcal{E}_8$. As mentioned earlier, we need to add at least two edges to make the neighborhood of a lonely vertex Hamiltonian. Since the Euler bound is $|E| \leq 3(15)-6+3(6) = 57$, we are looking for graphs in $\mathcal{E}_8$ with at most 55 edges. However, by Proposition \ref{prop-edgecount}, all such graphs belong to $\mathcal{D}_8$, which have exactly 55 edges and two lonely vertices. Since we need to add four edges to make the graph locally Hamiltonian, no such embedding can exist.
\end{proof}

\begin{theorem}
The $9$-critical graphs embeddable in the (nonorientable) surface of Euler genus 9 are $K_9$, $K_6 + C_5$, $K_5 + H_7$, and $K_5 + M_7$. 
\label{thm-euler9}
\end{theorem}
\begin{proof}
Similar to the previous proof, the Kostochka-Stiebitz and Gallai bounds rule out all cases except $|V| = 9, 11, 12$ and graphs in $\mathcal{E}_9$. By Theorems \ref{thm-mct-nonorient} and \ref{thm-nonorient}, $K_9$ and $K_6 + C_5$ embed in this surface. We also have triangular rotation systems for $K_5 + H_7$ (with an extra edge $(4, 5)$): 

$$\begin{array}{rccccccccccccccccccccccccccccccccc}
a. & b & 2 & c & 3 & d & 5 & e & 6 & 4 & 0 & 1 \\
b. & a & 1 & d & \TT{c} & 0 & 4 & 5 & 3 & 6 & e & 2 \\
c. & a & 2 & 1 & 5 & 4 & e & \TT{d} & \TT{b} & \TT{0} & 6 & 3 \\
d. & a & 3 & 4 & 6 & \TT{0} & \TT{e} & \TT{c} & b & 1 & 2 & 5 \\
e. & a & 5 & 1 & 0 & \TT{d} & c & 4 & 3 & 2 & b & 6 \\
0. & a & 4 & b & \TT{c} & \TT{6} & \TT{d} & e & 1 \\
1. & a & 0 & e & 5 & c & 2 & d & b \\
2. & a & b & e & 3 & 5 & d & 1 & c \\
3. & a & c & 6 & b & 5 & 2 & e & 4 & d \\
4. & a & 6 & d & 3 & e & c & 5 & b & 0 \\
5. & a & d & 2 & 3 & b & 4 & c & 1 & e \\
6. & a & e & b & 3 & c & \TT{0} & d & 4
\end{array}$$
and $K_5 + M_7$:
$$\begin{array}{rllllllllllllllllllllllllll}
a. & b & 3 & d & 6 & 1 & 2 & e & 4 & 0 & c & 5 \\
b. & a & 5 & 0 & d & \TT{c} & 1 & 6 & e & 2 & 4 & 3 \\
c. & a & 0 & 6 & 2 & 3 & 4 & e & \TT{d} & \TT{b} & \TT{1} & 5 \\
d. & a & 3 & 5 & \TT{1} & \TT{e} & \TT{c} & b & 0 & 4 & 2 & 6 \\
e. & a & 2 & b & 6 & 0 & 5 & 3 & 1 & \TT{d} & c & 4 \\
0. & a & 4 & d & b & 5 & e & 6 & c \\
1. & a & 6 & b & \TT{c} & \TT{5} & \TT{d} & e & 3 & 2 \\
2. & a & 1 & 3 & c & 6 & d & 4 & b & e \\
3. & a & b & 4 & c & 2 & 1 & e & 5 & d \\
4. & a & e & c & 3 & b & 2 & d & 0 \\
5. & a & c & \TT{1} & d & 3 & e & 0 & b \\
6. & a & d & 2 & c & 0 & e & b & 1.
\end{array}$$

By Proposition \ref{prop-edgecount}, the graphs in $\mathcal{E}_9$ have at least 71 edges, but the Euler bound is $|E| \leq 72$. Since graphs in $\mathcal{E}_9$ have a lonely vertex, no such embedding can exist. 
\end{proof}

\section{Conclusion}

We computed the genus of $K_n-C_5$ by constructing triangular embeddings of $K_n-P_3$, $K_n-(K_2 \cup P_2)$, and $K_n-P_4$ in orientable and nonorientable surfaces. The motivation of our result was to characterize the $(h(g)-1)$-critical graphs in the surfaces of Euler genus $g$, and the techniques used in the nonorientable part of the proof also yield a simpler solution to the minimum triangulations problem for nonorientable surfaces. 

If we extend the definition of $H(g)$ to $g = 0, 1$, then $h(0) = 4$ and $h(1) = 6$, coincidentally equaling the chromatic numbers of these surfaces. The odd cycles form infinitely many $(h(0)-1)$-critical planar graphs. As mentioned earlier, there are infinitely many 5-critical graphs embeddable in surfaces besides the sphere. Thus, there are infinitely many $(h(1)-1)$-critical projective planar graphs, as well.  For Euler genus $2$, $h(2)-1 = 6$, and, as mentioned earlier, the $6$-critical graphs on these surfaces have already been characterized \cite{Klein1, Klein2, Thomassen-Torus}. Thus, the remaining open cases are the surfaces of Euler genus 3 and 4. 

Recall that \v{S}krekovski \cite{Skrekovski} proved a more general result for $(h(g)-c)$-critical graphs, for any fixed positive integer $c$. In light of Proposition \ref{prop-kp3}, can we go one step further and prove a similar result $c = 2$? Finding genus embeddings of $K_n + H_7$ and $K_n + M_7$ is probably quite challenging in general: for certain residues of $n \pmod{12}$, these graphs are expected to have triangular embeddings, so we cannot always rely on Proposition \ref{prop-del}. 

\bibliographystyle{alpha}
\bibliography{biblio}

\appendix

\end{document}